\documentclass[10pt]{amsart}
\usepackage{amssymb}
\usepackage{epsfig}
\usepackage{url}
\usepackage{setspace}
\usepackage{lscape}
\usepackage{color}

\theoremstyle{plain}

\newtheorem{thm}{Theorem}[section]

\newtheorem{rem}[thm]{Remark}

\newtheorem{exam}[thm]{Example}
\def\cal{\mathcal}
\def\bbb{\mathbb}
\def\op{\operatorname}
\renewcommand{\phi}{\varphi}

\newcommand{\Z}{\bbb{Z}}
\newcommand{\Q}{\bbb{Q}}

\newcommand{\C}{\bbb{C}}

\begin{document}

\title[Constructions of diagonal quartic and sextic surfaces]{Constructions of diagonal quartic and sextic surfaces with infinitely many rational points}
\author{Andrew Bremner, Ajai Choudhry, and  Maciej Ulas}

\keywords{diagonal quartic surface, elliptic surface, rational points} \subjclass[2010]{Primary: 11G35, Secondary: 11D25, 11D41, 14G05}
\thanks{Research of the third named author was supported by Polish Government funds for science, grant IP 2011 057671 for the years 2012--2013.}

\begin{abstract}
In this note we construct several infinite families of diagonal quartic surfaces
\begin{equation*}
ax^4+by^4+cz^4+dw^4=0,
\end{equation*}
where $a,b,c,d\in\Z\setminus\{0\}$ with infinitely many rational points and satisfying the
condition $abcd\neq \square$. In particular, we present an infinite family of diagonal
quartic surfaces defined over $\Q$ with Picard number equal to one and possessing infinitely
many rational points. Further, we present some sextic surfaces of type
$ax^6+by^6+cz^6+dw^i=0$, $i=2$, $3$, or $6$, with infinitely many rational points.
\end{abstract}

\maketitle

\section{Introduction}\label{Section1}
R. van Luijk~\cite{Luijk1} recently solved an old problem attributed to D. Mumford by
constructing a quartic K3 surface with geometric Picard number equal to one. Moreover,
he was able to solve a question posed by P. Swinnerton-Dyer and B. Poonen at an AIM workshop
in December 2002, who asked whether
there exists a K3 surface over a number field with Picard number one that contains infinitely many
rational points. Van Luijk was able to show that the set of quartic surfaces with Picard number
one and infinitely many rational points is dense in the set of all K3 surfaces in the
Zariski topology.
Here, we consider a related problem. More precisely, we are interested
in finding the simplest possible form of an equation defining a quartic surface $\cal{S}$ over
$\Q$ with Picard number one, and possessing infinitely many rational points. Here, we are talking about
the Picard number $\rho(S)=\mbox{rank of the Picard group of }\cal{S}$, rather than the geometric
Picard number, which is defined as the Picard number of $\overline{S}=S\times_{\Q}\overline{\Q}$.
We believe that this is provided by a surface with diagonal equation of the form
\begin{equation}
\label{generalquar}
\cal{S}:\;ax^4+by^4+cz^4+dw^4=0,
\end{equation}
with $a,b,c,d\in\Z\setminus\{0\}$.
It is not clear a priori whether such an example exists, and we were unable to find
a suitable reference in the existing literature.  The big advantage of considering surfaces
in diagonal quartic form follows from the fact that the Picard number of $\cal{S}$ may be easily
computed.
Pinch and Swinnerton-Dyer~\cite{PiSwi}, as well as Swinnerton-Dyer~\cite{SwiD2},
investigate in some generality the arithmetic of the diagonal quartic surface,
and in particular, the former give tables allowing one to read off the Picard rank of any given
diagonal quartic surface over $\Q$. We thus need to find many (or better, infinitely many)
examples of diagonal surfaces with infinitely many rational points.
Having these examples in hand, we can check whether some have Picard number
equal to one.
The literature concerning the arithmetic of diagonal quartic surfaces is large and starts with
investigations of the surface $x^4+y^4=z^4+w^4$ considered by Euler (see Mordell~\cite{Mor} and
Dem'janenko~\cite{Dem}). Swinnerton-Dyer~\cite{SwiD1} studied the surface in detail, describing
all curves of arithmetic genus 0 lying on the surface.  A classical paper of Richmond~\cite{Rich}
shows how one can construct a new point on (\ref{generalquar}) from a given non-trivial point when
$abcd=\square$.  Bright has written extensively on the subject, starting with his
thesis \cite{Bri}.  Richmond's idea was used by Logan, van Luijk, and
MacKinnon~\cite{LoLuM} to prove density (in the Euclidean topology) of the rational points
on (\ref{generalquar}) when $abcd=\square$.
In a recent paper van Luijk~\cite{Luijk2} extended ideas of Swinnerton-Dyer \cite{SwiD3} and
proved that if the surface $V$ defined over a number field $k$ admits two elliptic fibrations
then there exists an explicitly computable closed subset $Z$ of $V$, not equal to $V$, such that
for each field extension $K$ of $k$ of degree at most $d$ over the field of rational numbers,
the set $V(K)$ is Zariski dense as soon as it contains any point outside $Z$.
In particular, his theorem implies that if the quartic surface $\cal{S}$ admits two elliptic
fibrations, the set $\cal{S}(\Q)$ is infinite, and a suitably chosen point lies outside the easily
computable set $Z$, then necessarily $\cal{S}(\Q)$ is dense in the Zariski topology. More
specifically, the elliptic fibrations do not necessarily need to
possess a section: it is enough to have one properly chosen specialization which is of positive rank.
We have (maybe somewhat arbitrarily) restricted attention to surfaces (\ref{generalquar})
where $abcd \neq \square$, since many existing results depend upon $abcd=\square$, and in which
situation the Picard number is at least 2.

We know relatively few examples of surfaces (\ref{generalquar}) with
infinitely many points, satisfying the condition $abcd\neq \square$. These include
the quartic $x^4+y^4+z^4=w^4$ (of Picard rank 4) considered by Elkies~\cite{Elk1}, and the
quartics $x^4+y^4+z^4=2t^4$, $x^4+y^4+4z^4=t^4$, $x^4+2y^4+2z^4=t^4$, $x^4+8y^4+8z^4=t^4$,
(of ranks 4, 8, 3, 3, respectively) in Carmichael~\cite{Carm}.
In fact, the method used by Carmichael also yields parametric solutions of the Diophantine equation
(\ref{generalquar}) whenever $a+b=0$ and $4a^2cd$ is a perfect fourth power. As this solution is not
given explicitly in \cite{Carm}, it is stated below:
\[
x=k(8as^4-ct^4), \quad y=k(8as^4+ct^4), \quad z=8kas^3t, \quad w=4acst^3
\]
where $s,\,t$ are arbitrary parameters and $k$ is determined by the relation $4a^2cd=k^4.$
Further, when
$(a,\,b,\,c,\,d)=(1,1,1,-2k^2)$, and the Diophantine equation $X_1^2+3X_2^2=k X_3^2$ is
solvable, a solution of (\ref{generalquar}) in terms of parameters $p,\,q$ is  given by
\[x = \phi_1(p,\,q)-\phi_2(p,\,q),\, y = 2\phi_2(p,\,q),\, z = \phi_1(p,\,q)+\phi_2(p,\,q),\, w = \phi_3(p,\,q),
\]
where $X_i=\phi_i(p,\,q),\,i=1,\,2,\,3,$ is a parametric solution of the equation $X_1^2+3X_2^2=k X_3^2$.

We also note that some believe a small Picard number of $\cal{S}$ has a strong influence on
the set of rational points on $\cal{S}$. This belief was questioned in a recent paper of
Elsenhans~\cite{Els}, who considered the set of surfaces $ax^4+by^4=cz^4+dw^4$  with
$1\leq a, b, c, d\leq 15$ from a computational point of view. He noted that in the considered
range there is no great difference between the number of unsolved cases for differing
Picard rank. In light of this
remark it is natural to look for examples of diagonal quartic surfaces with infinitely many
rational points and small Picard number. \\ \\
We briefly describe the content of the paper.
In Section \ref{sec2} the first construction of diagonal quartic surfaces with $abcd\neq \square$
is given. The construction is to intersect the quartic surface with a quadric surface chosen in
such a way that the intersection splits into quadric curves.
In Section \ref{sec3} the second construction is given. This method is equivalent to finding
curves of low degree on the affine variety $aX^4+bY^4+cZ^2+dW^2=0$, in particular curves
where $Z,W$ are of degree at most 2. Demanding then that $Z,W$ be squares leads to a condition
governed by a curve of genus 1 that may be chosen to have positive rank.
The third construction is presented in Section \ref{sec4}. It is based on
diagonal quartic curves (of genus 3) of the form $Px^4+Qy^4+Rz^4=0$ with a given rational point.
In Section \ref{sec5} we use the same idea as in Section \ref{sec3} by seeking curves of low
degree on affine varieties
of the form $Px^6+Qy^6+Rz^3+Sw^2=0$ and $Px^6+Qy^6+Rz^3+Sw^3=0$. These constructions allow us
to find infinitely many sextic surfaces of the form $Px^6+Qy^6+Rz^6+Sw^2=0$
and $Px^6+Qy^6+Rz^6+Sw^3=0$ with infinitely many points. This is motivated by earlier
research~\cite{BreUl2} concerning the existence of rational points on $a(x^2-y^6)=b(z^6-w^6)$.
Finally, in the last section we give some remarks concerning possible extensions of these
approaches.

\section{First construction} \label{sec2}

Consider a diagonal quartic surface $c_1 x^4+c_2 y^4+c_3 z^4+w^4 =0$,
and demand that its intersection with a quadric surface of the type
\[ w^2 = a_1 x^2+a_2 x y+a_3 x z+a_4 y^2+a_5 y z+a_6 z^2 \]
be singular, splitting into two quadrics (each with multiplicity two).
This is equivalent to the demand that
\begin{equation}
\label{intersectionquartic}
c_1 x^4+c_2 y^4+c_3 z^4+(a_1 x^2+a_2 x y+a_3 x z+a_4 y^2+a_5 y z+a_6 z^2)^2
\end{equation}
split into two quadratic factors.  We let
\begin{equation}
\label{quadfactor}
z^2 + (b_1 x +b_2 y) z + b_3 x^2 + b_4 x y + b_5 y^2
\end{equation}
be one of the quadratic factors.

Write out the system of nine equations resulting from equating to zero
the coefficients of the remainder on dividing (\ref{intersectionquartic})
by (\ref{quadfactor}). We successively eliminate variables developing
a branching tree of cases, according as to vanishing or non-vanishing of
denominators arising from the elimination. A full analysis is too tiresome:
Groebner bases of some ideals arising in this process contain many hundreds
of polynomials of high degree. In order to find non-trivial examples, we
solved in turn for $c_1,c_2,a_1,a_3,b_5,b_3,b_1$, leading to two equations
in $a_2,a_4,a_5,a_6,b_2,b_4,c_3$ to be satisfied. The denominators vanishing
provided four cases to consider. Otherwise, the numerators have respective degrees
$4$, $3$, in $a_2$, and the resultant with respect to $a_2$ contains eight factors
(excluding the factor $c_3$). We were able successfully to investigate seven of
these factors (the seventh is quadratic in $a_4$, allowing parameterization; the eighth
has degrees $4,8,12,8,6$ in $a_4,a_5,a_6,b_2,c_3$).   \\

\noindent
When $a_3=a_5=b_1=b_2=0$, essentially just two examples arise. \\
First, we obtain the surface
\[ V_1: \; \frac{a^2(s^2-2r t)}{s^2}x^4+\frac{a^2 (s^2-2 r t)t^2}{r^2 s^2}y^4+\frac{2a^2 t}{r s^2}z^4-w^4=0, \]
which when intersected with the quadric
\[ w^2 = a x^2+\frac{2 a t}{s}x y +\frac{a t}{r} y^2 \]
results in the identity
\begin{align}
\label{quadint1}
& 2a^2rt(rx^2 + sxy + ty^2 + z^2)(rx^2 + sxy + ty^2 - z^2) -  \nonumber \\
&  (arsx^2+2artxy+atsy^2+rsw^2)(arsx^2+2artxy+atsy^2-rsw^2) =  \nonumber \\
& a^2(-s^2+2r t)r^2 x^4+a^2 (-s^2+2 r t)t^2y^4-2a^2rtz^4+r^2s^2w^4.
\end{align}
The surface $V_1$ has generic Picard rank equal to 2.
We find surfaces $V_1$ with infinitely many points by demanding for example that the intersection
of the two quadrics
\[r x^2 + s x y + t y^2 + z^2 = 0, \quad a r s x^2+2a r t x y+a s t y^2+r s w^2 = 0, \]
represents an elliptic curve of positive rational rank.
If we take $a=r=-1$, then the intersection becomes
\[ -x^2 + s x y + t y^2 + z^2=0, \quad s x^2 + 2 t x y - s t y^2 - s w^2 = 0, \]
equivalently, on setting $2t=u s^2$, $(X,Y,Z,W)=(x,s y/2,z,w)$,
\begin{equation}\label{model1}
 X^2 - 2 X Y - 2 u Y^2 = Z^2, \quad X^2 + 2 u X Y - 2 u Y^2 = W^2.
 \end{equation}
This is an elliptic curve on taking $(X,Y,Z,W)=(1,0,1,1)$ as origin. The point
$Q(X,Y,Z,W)=(1,0,1,-1)$ is of infinite order, and the multiples $mQ$, $m=2,3,...$ give
infinitely many points on the surface
\begin{equation}
\label{ex1}
S: \; (1+u) X^4 + 4u^2(1+u) Y^4 - u Z^4 - W^4 = 0
\end{equation}
of generic Picard rank 2.
For example, when $m=2$ we obtain
\[ (X,Y,Z,W)=(1+10u+u^2, \; -4(1-u), \; -3-10u+u^2, \; 1-10u-3u^2), \]
and when $m=3$
\begin{align*}
(X,Y,Z,W)= & (-(1+10u+u^2)(1+52u+38u^2+52u^3+u^4),  \\
           &  4(1-u)(-3-10u+u^2)(-1+10u+3u^2), \\
           &  -5-114u+233u^2+1140u^3+381u^4+94u^5-u^6, \\
           &  1-94u-381u^2-1140u^3-233u^4+114u^5+5u^6);
\end{align*}
etc. etc. \\
In light of the identity (\ref{quadint1}), the surface $S$ at (\ref{ex1}) contains the
pair of fibrations
\[ Q_1 = \lambda Q_3, \quad \lambda Q_2 = Q_4, \qquad \mbox{ and } \qquad Q_1 = \lambda Q_4, \qquad \lambda Q_2 = Q_3, \]
where
\begin{align*}
Q_1=u (-X^2+2X Y+2u Y^2+Z^2), \quad & Q_2=s(-X^2+2X Y+2u Y^2-Z^2), \\
Q_3=-X^2-2u X Y+2u Y^2+W^2, \quad & Q_4=s(X^2+2u X Y-2 u Y^2+W^2).
\end{align*}
The existence of infinitely many rational points, together with the
result of van Luijk (and computations to ensure the applicability of van Luijk's result), now
implies the following:

\begin{thm}\label{pic}
The set of diagonal quartic surfaces defined over $\Q$ with Picard number 2 and Zariski
dense set of rational points is infinite.
\end{thm}
\begin{proof} Due to the large number of parameters it is a difficult task to perform all the necessary
computations to apply Van Luijk's theorem to the surface $V_{1}$. We thus present the proof in the
simplest possible case. More precisely, consider the surface $S$ at (\ref{ex1}), which is a special case
of the surface $V_{1}$. The surface $S$ contains a genus one curve, say $C$, defined by the intersection
of the two quadrics at (\ref{model1}). The curve $C$ is smooth provided that $u\not\in\{-2,-1,-1/2,0\}$.
This model contains the rational point $(1,0,1,1)$, which can be used as a point at infinity.
Then there exists a birational map $\phi: \; C \rightarrow \cal{E}$, where $\cal{E}$ is the cubic curve over
$\Q(u)$ with Weierstrass equation
\begin{equation*}
\cal{E}:\;Y^2=X^3+27 u(6+11u+6u^2)X+54u^2(18 + 37 u + 18 u^2).
\end{equation*}
Note that the point $(1,0,1,-1)$ does not lie on any of the 48 lines lying on $S$. We need to check
that $(1,0,1,-1)$ corresponds to a non-torsion point on $\cal{E}$. Indeed, the curve $\cal{E}$ contains the
point $U=\phi((1,0,1,-1))=(3u(3u+4),27u(u+1)(u+2))$. By specialization at $u=2$, the point $U_{2}$, lying on
the curve $\cal{E}_{2}$, is of infinite order on $\cal{E}_{2}$ (for $2U_{2}$ is not integral, and the
Lutz-Nagell theorem applies).
This immediately implies that $U$ is of infinite order on $\cal{E}$. However, we are interested in
characterizing those $u=u_{0}\in\Q$ such that the specialized point $U_{u_{0}}$ is of infinite order on
$\cal{E}_{u_{0}}$. By Mazur's theorem we know that the order of a torsion point on an elliptic curve defined over
$\Q$ is at most $12$. A quick computation of the multiplicities $mU$ for $m=1,2,\ldots, 12$, reveals that
$mU_{u_{0}}=\cal{O}$ only for $m=4$ with $u_{0}=1$. Thus for $u_{0}\in\Q\setminus\{-2,-1,-1/2,0,1\}$ the point
$U_{u_{0}}$ is of infinite order on $\cal{E}_{u_{0}}$. Applying now Van Luijk's theorem, we get that for each
$u\in\Q\setminus\{-2,-1,-1/2,0,1\}$ the set of rational points on $S$ is Zariski dense.
\end{proof}

\begin{rem}
{\rm The reason why the above proof is simple is clear. We have an explicit point (in our case $(1,0,1,-1)$)
which lies on the surface, and thus on the corresponding elliptic curve (in our case $\cal{E}$). Thus, all
computations can be made explicit. In the general case, on the surface $V_{1}$ we do not have a generic
rational point and thus reasoning depends on the possible form of the required point(s) with (unknown)
coefficients $P=(X_{0},Y_{0},Z_{0},W_{0})$. In particular, it is very difficult to characterize the
torsion subgroup of those smooth fibers containing the point $P$.}
\end{rem}

%
%

\noindent
Second, we obtain the surface
\[ V_2: \; a^2 x^4+d^2 y^4 +2 a d s^2 z^4 - w^4 = 0, \]
which when intersected with the quadric
\[ w^2 = a x^2 + d y^2 \]
results in the identity
\begin{align*}
& (a x^2+d y^2 +w^2)(a x^2+d y^2-w^2) -2 a d (x y - s z^2) (x y + s z^2) = \\
& a^2 x^4+d^2 y^4 +2 a d s^2 z^4 - w^4.
\end{align*}
The surface has generic Picard rank 2.
We find surfaces $V_2$ with infinitely
many points by demanding for example that the intersection of the two quadrics
\[ a x^2+d y^2 = w^2 , \quad x y = s z^2 \]
represents an elliptic curve of positive rational rank.
Set $(x,y,z,w)=(s X,Y,Z,s u W)$, and take $(a,d)=(\alpha u^2, (1-\alpha)s^2 u^2)$,
with $(X,Y,Z,W)=(1,1,1,1)$ as elliptic curve origin.
The point $Q(X,Y,Z,W)=(1,1,-1,1)$ is of infinite order and the multiples $mQ$, $m=2,3,...$
give infinitely many points on the surface of generic rank 2
\begin{equation}
\label{ex2}
\alpha^2 X^4 + (1-\alpha)^2 Y^4 +2 \alpha(1-\alpha)Z^4 - W^4 = 0.
\end{equation}
For example with $m=2$:
\begin{align*}
(X,Y,Z,W) = & ((-3+4\alpha^2)^2, \; (1-8\alpha+4\alpha^2)^2, \; (-3+4\alpha^2)(1-8\alpha+4\alpha^2), \\
 & 1+24\alpha-72\alpha^2+96\alpha^3-48\alpha^4);
\end{align*}
etc. \\

\noindent
The only resulting intersection of the form we seek arising from non-vanishing denominators
is the following. The surface
\[ V_3: \; x^4+y^4 - (2u^2-v^2)^2 z^4 + (2u^2-v^2) k^2 w^4 = 0 \]
intersected with the quadric surface
\[ (2u-v) k w^2 = (x+y)^2 + 2(u-v) (x+y) z - (2u^2-v^2) z^2 \]
results in the identity
\begin{align*}
& 2\left( (3u-2v)(x^2+y^2) + (4u-3v)x y - (2u^2-v^2)(x+y+(u-v)z)z \right) \times \\
& \left( u(x^2+y^2) + v x y + (2u^2-v^2)(x+y+(u-v)z)z \right)- \\
&  (2u^2-v^2) \left( (x+y)^2 + 2(u-v)(x+y)z - (2u^2-v^2)z^2 - (2u-v)k w^2 \right) \times \\
& \left( (x+y)^2 + 2(u-v)(x+y)z - (2u^2-v^2)z^2 + (2u-v)k w^2 \right) \\
& = (2u-v)^2 \left( x^4+y^4 - (2u^2-v^2)^2 z^4 + (2u^2-v^2) k^2 w^4 \right).
\end{align*}
The surface $V_3$ has generic Picard rank equal to 3.
To find surfaces $V_3$ with infinitely many rational points,
it suffices for example to find infinitely many rational points on the intersection
\begin{align*}
F_0: & \; (x+y)^2 + 2(u-v) (x+y) z - (2u^2-v^2) z^2 = (2u-v) k w^2, \\
F_1: & \; u (x^2+y^2) + v x y + (2u^2-v^2)(x+y+(u-v)z)z = 0.
\end{align*}
We set ($x,y,z)=(p,q,r)$ and regard the equation of $F_1$ as defining a cubic curve
$C_1$ in the $u,v$-plane over ${\Q}(p,q,r)$, containing the rational point $(u,v)=(0,0)$.
As such, it is an elliptic curve with cubic model:
\[ V^2 = U(U^2 -(p^4+q^4)). \]
We seek points $(u,v)$ on $C_1$ such that $2u^2-v^2 \neq \pm \square$, in order to
keep small the Picard rank of the surface. Now the points $P_1=(-p^2,-p q^2)$,
$P_2=(-q^2,-p^2 q)$ on the cubic model are independent, and $P_1+P_2$ pulls back
to the point on $C_1$
\[ P(u,v) = (-p q(p+q)/((p^2+p q+q^2)r), (p+q)(p^2+q^2)/((p^2+p q+q^2)r)), \]
with
\[ 2u^2-v^2=-(p+q)^2(p^4+q^4)/((p^2+p q+q^2)^2r^2). \]
For this point $P$ to determine a point on the quadric $F_0$, we require
\[ k=p q r(2 p^2+3 p q+2 q^2)/((p+q)(p^2+p q+q^2)).  \]
On setting $z=(p^2+p q+q^2)r Z /(p+q)$, $w=(p^2+p q+q^2)W$, the surface now takes the form
\begin{equation}
\label{surface1}
V: x^4+y^4 - (p^4+q^4)^2 Z^4 - p^2 q^2 (2 p^2+3 p q+2 q^2)^2(p^4+q^4)W^4 = 0;
\end{equation}
the generic Picard rank is 3.  The intersecting quadric becomes
\[ F_0: \; p q(p+q)^2 (2p^2+3 p q+2 q^2) W^2 = -(x+y)^2 +2(p^2+p q+q^2)(x+y)Z-(p^4+q^4)Z^2; \]
and the intersection comprises the two quadrics
\begin{align*}
F_1: \; & -p q(x^2+y^2) + (p^2+q^2)x y - (p^4+q^4)(x+y) Z + (p^2+p q+q^2)(p^4+q^4) Z^2 = 0, \\
F_2: \; & (2p^2+3p q+2q^2)(x^2+y^2) + (3p^2+4p q+3q^2)x y - (p^4+q^4)(x+y)Z \\
     & +(p^2+p q+q^2)(p^4+q^4)Z^2 = 0.
\end{align*}
We have arranged that $F_1$ contain the point
$(x,y,Z)=(p,q,(p+q)/(p^2+p q+q^2))$, which allows parameterization of all points
on $F_1$ by:
\begin{align*}
x:y:z = & p q^2(p^2+q^2) a^2 + (p+q)(p^4+q^4) a b +p(p^4+q^4) b^2: \\
& p^2 q(p^2+q^2) a^2 + (p+q)(p^4+q^4) a b +q(p^4+q^4) b^2: \\
& (p-q)^2(p+q) a b.
\end{align*}
The inverse is given by
\[ a/b = -((p-q)q x - (p-q)p y + (p^4+q^4) z)/(p q (p^2+q^2) z). \]
Such a point gives rise to a rational point on the intersecting quadric $F_0$
precisely when
\begin{align*}
& p q(2p^2+3p q+2q^2) W^2 = -p^2 q^2(p^2+q^2)^2a^4-2p q(p+q)^2(p^2+q^2)(p^2-p q+q^2)a^3b\\
& -(p^4+q^4)(p^4+2p^3q+6p^2q^2+2p q^3+q^4)a^2b^2-2(p+q)^2(p^2-p q+q^2)(p^4+q^4)a b^3 \\
& -(p^4+q^4)^2b^4
\end{align*}
and we have arranged that there is a point on this elliptic quartic at
\[ (a,b,W)=(p^4+q^4, -p q(p^2+q^2), p q(p-q)^2(p^2+q^2)(p^4+q^4)). \]
This latter point is of infinite order, and hence its multiples
pull back to infinitely many rational points on (\ref{surface1}).
Using the same type of reasoning as in the proof of Theorem \ref{pic}, we get the following.

\begin{thm}\label{pic3}
 The set of diagonal quartic surfaces defined over $\Q$ with Picard number 3 and Zariski
dense set of rational points is infinite.
\end{thm}
\noindent

\begin{rem} {\rm The essence of the method described in this section was used in \cite{BreUl1}
to find infinitely many integral points on the surface $w^2=x^6+y^6+z^6$.}
\end{rem}
\begin{rem}{\rm The approach of this section cannot find diagonal quartic surfaces with Picard number
one, for the construction immediately implies the existence of two elliptic fibrations on
$\cal{S}$ (defined over $\Q$) and thus the Picard number is at least two (page 32 of Bright \cite{Bri}, Corollary 2.27).}
\end{rem}

\begin{rem}
{\rm A nice illustration of Van Luijk's theorem in practice can be also found in the Master's thesis of
D. Festi \cite{Festi}.}
\end{rem}

%
%
%

\section{Second construction} \label{sec3}

The second method we describe to construct diagonal quartics with infinitely many rational points
is even more direct than the previous one. Consider the genus three curve $C:\;Px^4+Qy^4+Rz^4=0$
and suppose that $(a,b,c)$, with $abc\neq 0$ is a rational point lying on $C$.  Substituting
\begin{equation*}
x=a T, \quad y=-\frac{c^3R}{b^3Q}+b T, \quad z=1+c T,
\end{equation*}
there results
\begin{equation*}
G(T)=\frac{-a^4 P R}{b^{12}Q^3}(b^8Q^2-b^4c^4Q R+c^8R^2 + 4b^4c Q(b^4Q-c^4R)T + 6b^8c^2Q^2T^2)=0.
\end{equation*}
The polynomial $G$ is of degree two in $T$ and thus if a rational number $S$ is chosen so that the
quartic curve
\begin{equation*}
\cal{C}:\;-Sw^4=G(T)
\end{equation*}
has infinitely many rational points, then there will be infinitely many rational points on the surface
$Px^4+Qy^4+Rz^4+Sw^4=0$.

To construct some examples, let $a, b, c$ be integer parameters and
let $R=-(Pa^4+Qb^4)/c^4$ be a non-zero rational number. Then the point $(a,b,c)$ lies on $C$ and
we seek $S$ such that there is a point of infinite order on the curve
\begin {align*}
\cal{C}: \; & -SW^4 = H(T) \\
          = & P Q(a^4P+b^4Q)(a^8P^2+3a^4b^4P Q+3b^8Q^2 + 4b^4c Q(a^4P+2b^4Q)T + 6b^8c^2Q^2T^2),
\end{align*}
where $W=\frac{b^3c Q}{a}w$. In order to guarantee the existence of a point on $\cal{C}$
we put $S=-H(t)$, so that $\cal{C}$ contains the point $(W,T)=(1,t)$.

\begin{exam}{\rm
To produce an explicit example, put $R=-P-Q$, $a=b=c=1$. Thus we consider the curve
$-S w^4= P Q(P+Q)(P^2+3P Q+3Q^2 + 4 Q(P+2Q)T + 6Q^2T^2)$ with $S=-P Q(P+Q)(P^2+3P Q+3Q^2+4Q(P+2Q)t+6Q^2t^2)$.
That is, the curve
\begin{align*}
\cal{C}: \;  & P Q(P+Q)(P^2+3P Q+3Q^2+4Q(P+2Q)t+6Q^2t^2) W^4 = \\
             & P Q(P+Q)(P^2+3P Q+3Q^2 + 4 Q(P+2Q)T + 6Q^2T^2).
\end{align*}
Using the point $(W,T)=(1,t)$ as a point at infinity, we find that $\cal{C}$ is birationally
equivalent to the elliptic curve defined over $\Q(P,Q)(t)$ with cubic model:
\[ \cal{E}: \; Y^2 = X(X^2 + 3(P^2+P Q+Q^2)(P^2+3P Q+3Q^2+4P Q t+8Q^2 t+6Q^2t^2)).  \]
From a geometric point of view,  $\cal{E}$ for fixed rational $P,Q$ represents a rational elliptic surface
in the plane $(X,Y,t)$.  Thus by results of Shioda, the generators of $\cal{E}(\C(t))$
lie among polynomial solutions $(X,Y)$ of the equation of $\cal{E}$ which satisfy the
conditions $\op{deg}_{t}X\leq 2,\;\op{deg}_{t}Y\leq 3$.  It is straightforward in any given instance
to compute all such polynomials defined over $\Q$.

For generic $P,Q$, the point $(W,T)=(-1,t)$ is of finite order on $\cal{C}$ precisely when $P+2Q+3 Q t=0$.
When $t \neq -(P+2Q)/(3Q)$ therefore, this point is of infinite order, and its multiples pull
back to infinitely many rational points on the diagonal quartic surface
\[ P x^4 + Q y^4 -(P+Q) z^4 -P Q(P+Q)(P^2+3P Q+3Q^2+4Q(P+2Q)t+6Q^2t^2) w^4 = 0. \]
The generic Picard rank is in this instance equal to 1. \\
To refine the example further, consider the specific case $P=Q=1$, and set $t=-1+u$.
Computation shows that for $u \neq 0$,
the point $M(-1,-1+u)$ is a generator for $\cal{E}(\Q(u))$. Its multiples $kM$, $k=1,2,3...$, pull back to
infinitely many rational points $(x,y,z,w)$ on the diagonal quartic surface
\begin{equation}
\label{St}
\cal{S}_{u}: \; X^4+Y^4-2Z^4-2(1+6u^2)W^4=0.
\end{equation}
The point $M$ pulls back to $(X,Y,Z,W)=(-1+u,1+u,u,-1)$; and $2M$ pulls back to
\begin{align*}
(X,Y,Z,W) =  & (-1 - 3u - 36u^3 + 54u^4 - 162u^5 - 324u^7 - 729u^8 + 81u^9, \\
& 1 - 3u - 36u^3 - 54u^4 - 162u^5 - 324u^7 + 729u^8 + 81u^9, \\
& 3u(-1 - 12u^2 - 54u^4 - 108u^6 + 27u^8), \\
& -(1 + 12u^2 + 9u^4)(-1 + 27u^4)),
\end{align*}
etc.
}
\end{exam}
\noindent
To our knowledge this is the first explicit example of a family of diagonal quartics each with
Picard number equal to 1 and possessing infinitely many rational points. We thus get:
\begin{thm}\label{pic4}
The set of diagonal quartic surfaces defined over $\Q$ with Picard number 1 and infinitely many
rational points is infinite.
\end{thm}
\section{Third construction } \label{sec4}
In this section we propose a different idea, which we name the polynomial method. More precisely,
we are interested in finding polynomial curves of low degree which lie on the affine variety
$F(x,y,z,w)=0$, where
\begin{equation*}
F(X,Y,Z,W)=PX^4+QY^4+RZ^2+SW^2.
\end{equation*}
Throughout, we suppose $PQRS \neq 0$, and that $PQRS$ is not a perfect square.
Essentially, we seek polynomial curves, say in the variable $T$, which satisfy the conditions
$\op{deg}X\leq 1,\op{deg}Y\leq 1$ and $\op{deg}Z\leq2, \op{deg}W\leq 2$.
It is clear that if such an identity exists, and the curve of genus one given by the intersection
\begin{equation*}
\alpha U^2 =Z(T),\quad \beta V^2=W(T)
\end{equation*}
has infinitely many rational points, then there are infinitely many rational points on the diagonal
quartic surface
$PX^4+QY^4+R\alpha^2 U^4+S\beta ^2 V^4=0$. \\ \\
If $\op{deg}X=0$, we have one of two identities:
\[ (P,Q,R,S)=(a c(a d-b c), -b d(a d-b c), -c d, a b), \quad (X,Y,Z,W)=(1,t,a+b t^2, c+d t^2), \]
where the product of coefficients is a square; and
\begin{align}
\label{ident0}
(P,Q,R,S)=((2a c-b^2)a^2, (2a c-b^2)c^2, 1, -2a c), & \\
(X,Y,Z,W)=(1,t, a b +2 a c t+b c t^2, a+b t+c t^2). \nonumber
\end{align}
If  $\op{deg}X=1$, then by homogeneity, by multiplying $T$ by a scalar, and by absorption of fourth powers
into $P,Q,R,S$, it suffices only to consider polynomial solutions of $F(X,Y,Z,W)=0$ of the form
\begin{equation*}
X=T, \quad Y=1+T, \quad Z=r+q T+T^2, \quad W=w+v T+T^2.
\end{equation*}
Equating the coefficients of the resulting expression in $T$ to zero, there arises:
\begin{equation*}
\begin{cases}
Q + r^2R + w^2 S=0,\\
2Q + q r R + 2v w S=0, \\
6Q + q^2R + 2r R + (v^2 + 2w)S=0,\\
2Q + q R + v S=0, \\
P + Q + R + S =0.
\end{cases}
\end{equation*}
Provided $r(q-2r) \neq 0$, then the first, second, and fifth equation imply that
\begin{align*}
P= & (q r - 2r^2 - v w + r^2v w + 2w^2 - q r w^2)S/(r(-q + 2r)), \\
Q= & (r v - q w)w S/(q - 2r), \\
R= & (v - 2w)w S/(r(-q + 2r))
\end{align*}
Provided $w(2r^2-q)+r(q-2r) \neq 0$, then the third and fourth equations give
\begin{align*}
v = & 2q(r-1)w^2/(w(2r^2-q)+r(q-2r)), \\
 & (r-w)(w-1)w(q r-2r^2-q w+q^2w-2q r w+2r^2w)=0.
\end{align*}
The instances $(r-w)(w-1)w=0$ lead to $PQ=0$; so necessarily
\[ w(-q+q^2-2q r+2r^2)+r(q-2r) = 0. \]
Now $-q+q^2-2q r+2r^2 \neq 0$, otherwise $r(q-2r)=0$; so that
\[ w=r(2r-q)/(-q+q^2-2q r+2r^2). \]
Together, we obtain the identity $P X^4+Q Y^4+R Z^2+S W^2=0$, where
\begin{align}
\label{ident1}
(P,Q,R,S) = & ( -(q-r-1)^2(q^2-2q r+2r^2-2r), \; -r^2(q^2-2q r+2r^2-2r), \\
            & 2(q-r-1)r, \; 1 ) \nonumber
\end{align}
and $(X,Y,Z,W)=$
\[ \left( T, \; 1+T, \; r+q T+T^2, \; -(q-2r)r+2(r-1)r T+(q^2-2q r-q+2r^2)T^2 \right). \]
The exceptions above can easily be dealt with:
if $r(q-2r) \neq 0$ and $w(q-2r^2)-r(q-2r)=0$, then $w=r(q-2r)/(q-2r^2)$, leading to
$q(r-1)=0$.  Now $r \neq 1$, otherwise $P=0$; so necessarily $q=0$ and $(v-2)(r v-r+1)=0$.
The case $v=2$ leads to $R=0$; and thus $v=(r-1)/r$. This gives the identity
$P X^4+Q Y^4+R Z^2+S W^2=0$,
where
\begin{equation}
\label{ident2}
(P,Q,R,S)=\left( (1-r)(1+r)^2, \; (1-r)r^2, \; -1-r, \; 2r \right)
\end{equation}
and
\[ (X,Y,Z,W) = \left( T, \; 1+T, \;r+T^2, \; r+(-1+r)T+r T^2 \right). \]
Finally, if $r(q-2r)=0$, then the only non-degenerate identity that results is
$P X^4+Q Y^4+R Z^2+S W^2=0$, where
\begin{equation}
\label{ident3}
(P,Q,R,S)=\left( (1-w)^2, \; w^2, \; 2w(1-w), \; -1 \right),
\end{equation}
and
\[ (X,Y,Z,W) = \left( T, \; 1+T, \; T+T^2, \; w+2w T+T^2 \right). \]
In the identities at (\ref{ident0}), (\ref{ident1}), (\ref{ident2}), (\ref{ident3}), it is straightforward
by choice of parameter to force rational points on to the intersection of the two quadrics $Z,W$
in the variable $T$,
thereby providing examples of diagonal quartic surfaces with infinitely many points.
The generic Picard rank in each case is equal to two. But the examples are interesting in their own right,
and we give a few more details.
\begin{exam}{\rm
At (\ref{ident0}), we demand
\[  a b +2a c t+b c t^2 = \alpha u^2, \quad a+b t+c t^2 = \beta v^2. \]
If we set, for example, $(\alpha,\beta)=(a b,a)$, then the intersection
represents an elliptic curve with origin at $(t,u,v)=(0,1,1)$. The point
$Q(t,u,v)=(0,1,-1)$ is of infinite order, and pullbacks of its
multiples $mQ$ provide infinitely many points on the surface of generic rank 2
\[ (2a c-b^2)a^2 X^4+(2 a c-b^2)c^2 Y^4+ a^2 b^2 Z^4 = 2 a^3 c W^4. \]
When $m=2$ for instance,
\begin{align*}
(X,Y,Z,W)= & (b^4 - 20a b^2c + 4a^2c^2, \; 8a b(b^2 + 2a c),  \\
& -b^4 - 20a b^2c + 12a^2c^2, -3b^4 + 20a b^2c + 4a^2c^2).
\end{align*}
}
\end{exam}
\begin{exam}{\rm
On homogenizing (\ref{ident1}), we have the example
\[ (p-q+r)^2(2p r-q^2+2q r-2r^2)X^4 + r^2(2p r-q^2+2q r-2r^2)Y^4 - 2r(p-q+r)s^2 Z^4 + t^2 W^4 = 0, \]
with $X=T$, $Y=1+T$, and where we demand
\[ r+ q T+p T^2 = s Z^2, \qquad r(2r-q) - 2r(p-r)T + (-p q+q^2-2q r+2r^2)T^2 = t W^2. \]
It is only necessary to specialize $p,q,r,s,t$ so that this latter intersection
represents an elliptic curve of positive rank. This occurs for example when
$s=r$, $t=r(2r-q)$, where we take $(T,Z,W)=(0,1,1)$ as zero of the corresponding
elliptic curve. The point $Q(T,Z,W)=(0,1,-1)$ is of infinite order, and pullbacks of its
multiples $mQ$ provide infinitely many points on the surface
\begin{align}
\label{surf1}
(p-q+r)^2(q^2-2 p r-2 q r+2r^2)X^4 & + r^2(q^2-2 p r-2 q r+2r^2)Y^4+ \\
2r^3(p-q+r)Z^4 & - (q-2r)^2r^2W^4 = 0. \nonumber
\end{align}
For example when $m=2$:
\begin{align*}
(X,Y, & Z,W) = (8(q-2r)r(q^2+2p r-6q r+6r^2), \\
& q^4-20p q^2r+12q^3r+4p^2r^2+72p q r^2-72q^2r^2-72p r^3+120q r^3-60r^4, \\
&-3q^4+20p q^2 r+4q^3r+4p^2r^2-88p q r^2+32q^2r^2+88p r^3-72q r^3+36r^4, \\
& q^4+20p q^2r-28q^3r-12p^2r^2-56p q r^2+112q^2r^2+56p r^3-168q r^3+84r^4).
\end{align*}
}
\end{exam}
\begin{exam}{\rm
On homogenizing (\ref{ident2}), we have the example
\[ -(r-s)(r+s)^2 X^4 - r^2 (r-s) Y^4 - (r+s) p^2 Z^4 + 2r q^2 W^4 = 0, \]
with $X=T$, $Y=1+T$, and where we demand
\[ r+s T^2 = p Z^2, \qquad  r + (r-s)T + r T^2 = q W^2. \]
It suffices to specialize $p,q,r,s$ so that this latter intersection
represents an elliptic curve of positive rank. This occurs for example when $p=q=r$,
with $(T,Z,W)=(0,1,1)$ as origin of the elliptic curve. The point $Q(T,Z,W)=(0,1,-1)$
is of infinite order, and pullbacks of its multiples $mQ$ provide infinitely many
points on the surface
\begin{align}
\label{surf2}
-(r-s)(r+s)^2 X^4-r^2(r-s) Y^4-r^2(r+s) Z^4+2r^3 W^4 = 0.
\end{align}
For example when $m=2$:
\[ (X,Y,Z,W)=(-8r(3r+s), \; -15r^2-18r s+s^2, \; 9r^2+22r s+s^2, \; 21r^2+14r s-3s^2 ); \]
etc.
}
\end{exam}
\begin{exam}{\rm
On homogenizing (\ref{ident3}), we have the example
\[ -(v - w)^2 X^4 - w^2 Y^4 - 2(v-w)w t^2 Z^4 + u^2 W^4 = 0, \]
with $X=T$, $Y=1+T$, and where we demand
\[ T+T^2 = t Z^2, \qquad  w + 2 w T + v T^2 = u W^2. \]
It is only necessary to specialize $t,u,v,w$ so that this latter intersection
represents an elliptic curve of positive rank. As illustration, take $v/w=2a$, $t=4(1-a)$, $u=w$,
with $(T,Z,W)=(0,0,1)$ as origin for the elliptic curve. The point $Q(T,Z,W)=(-1/a,1/(2a),1)$
is of infinite order, and pullbacks of its multiples $mQ$ provide infinitely many
points on the surface
\begin{align}
\label{surf3}
(1-2a)^2 X^4 + Y^4 - 2(1-a)^2(1-2a) Z^4 - W^4 = 0.
\end{align}
For example when $m=2$:
\[ (X,Y,Z,W) = (4(1-a)a^2, \; (2-4a+a^2)^2, \; 2a(2-4a+a^2), \; -4+16a-20a^2+8a^3+a^4), \]
etc.
}
\end{exam}
Finally, we observe that it is possible, given $m \in \Z$, to obtain surfaces of this
type with $PQRS \equiv m$ modulo squares. For set $r=1$, $q=2+2m$ in (\ref{ident1}).
There results
\[ 4m^2(1+m) T^4+(1+m) (1+T)^4-m \left( 1-(1+2m)T^2 \right)^2-\left( 1+2(1+m)T+T^2 \right)^2=0, \]
and the product of the coefficients is $m$ modulo squares.
We now demand that
\[ 1-(1+2m)T^2 = \square, \qquad 1+2(1+m)T+T^2 = \square. \]
This intersection contains a point at $T=0$, and the curve is of positive rank.
Some points are given by $T=$
\[ \frac{4(-1+m)}{(1+m)(5+m)}, \quad \frac{4(-1+m) (-3-10m+m^2)(-1+10m+3m^2)}{(1+m)(13+22m+m^2)(1-5m+27m^2+m^3)}, \dots \]
These give the surfaces
\[ 4m^2(1+m)X^4 + (1+m)Y^4 - m Z^4 - W^4 = 0, \]
with infinitely many points given by
\[ (X,Y,Z,W)= \left( 4(-1+m), \; 1+10m+m^2, \; -3-10m+m^2, \; -1+10m+3m^2 \right); \]
\begin{align*}
(X,Y,Z,W) = & (4(-1+m)(3+10m-m^2)(1-10m-3m^2), \\
& (1+10m+m^2)(1+52m+38m^2+52m^3+m^4), \\
& 5+114m-233m^2-1140m^3-381m^4-94m^5+m^6, \\
& 1-94m-381m^2-1140m^3-233m^4+114m^5+5m^6);
\end{align*}
etc. We thus proved:

\begin{thm}
{For any $m\in\Z$ there is a diagonal quartic surface
$ax^4+bY^4+cZ^4+dT^4=0$ containing infinitely many (non-trivial)
rational points and satisfying the property $abcd\equiv m\pmod{\Q^{2}}$.}
\end{thm}

\section{Identities for sixth powers} \label{sec5}
In this section we apply the polynomial method to find low degree curves on the two types of surface
that are of interest to us:
\begin{equation}\label{quadcube}
PX^6+QY^6+RZ^3+SW^2=0
\end{equation}
and
\begin{equation}\label{quadquad}
PX^6+QY^6+RZ^3+SW^3=0.
\end{equation}

First, consider the equation (\ref{quadcube}). We seek polynomial solutions of (\ref{quadcube}) of the form
\begin{equation*}
X=T,\quad Y=1+T,\quad Z=T^2+aT+b,\quad W=T^3+cT^2+dT+e.
\end{equation*}
Let $F(X,Y,Z,W)$ denote the left hand side of (\ref{quadquad}). Then after the above substitution
we obtain a degree six polynomial in $T$, say $F(X,Y,Z,W)=\sum_{i=0}^{6}c_{i}T^i$.
We are interested in solving in rationals the system
of equations $c_{i}=0$ for $i=0,\ldots,6$, with the assumption $PQRS\neq 0$. First, solve the
system
\begin{equation*}
c_{4}=c_{5}=c_{6}=c_{7}=0
\end{equation*}
with respect to $P, Q, d, e$. There results:
\begin{align*}
P&=\frac{1}{6}(3(a-2) R + 2(c-3)S),\\
Q&=-\frac{3aR+2cS}{6},\\
d&=\frac{3 (2a^2-5a + 2 b)R+2(c-5)cS}{4S},\\
e&=\frac{-3(2a^3-20a + 12ab+15ac-6a^2c-6bc)R+2c(3c^2-15c+20)}{12S}.
\end{align*}
Substituting into the first three equations and clearing the common denominator $144S$, we are left
with the system
\begin{equation*}
c_{1}(a,b,c,R,S)=c_{2}(a,b,c,R,S)=c_{3}(a,b,c,R,S)=0.
\end{equation*}
These equations are of degree $3, 2, 2$ respectively in the variable $b$. The polynomial
$M_{1}=\op{Res}_{b}(c_{1},c_{2})=20736 R^2 S \times W_{1}(a,c,R,S)$, with $\op{deg}_{R}W_{1}=7$;
and $M_{2}=\op{Res}_{b}(c_{2},c_{3})=5184 R^2 S\times W_{2}(a,c,R,S)$, with $\op{deg}_{R}W_{2}=5$.
Finally, $W(a,c,S)=\op{Res}_{R}(W_{1},W_{2})$ has the following factorization:
\begin{align*}
W(a,c,S)=-2^{41}&3^{36}5^{2}S^{35}(a-2)^{11}(a-1)^{3}a^{13}(3a-2c)^{4}(c-3)^{7}c^{8} \times \\
& (4-4c+a(c+1))(2c+a(c-4))f_{10}(a,c)f_{11}(a,c)f_{12}(a,c)f_{13}(a,c),
\end{align*}
where
\begin{align*}
&f_{10}(a,c)=(a^3-7a^2+23a+1)c^2+(a+1)(a^2-16a-27)c+2a(2a^2+14a+27),\\
&f_{11}(a,c)=(a^3+a^2+7a-27)c^2-(7 a^3+15a^2-49a+3)c+4a (2a-1)^2,\\
&f_{12}(a,c)=\sum_{i=0}^{6}p_{i}(a)c^{i},\\
\end{align*}
with
\begin{align*}
&p_{0}=8 a^3 (1+18 a+2058 a^2+1584 a^3+444 a^4+72 a^5+8 a^6),\\
&p_{1}=-12 a^2 (1+49 a+1601 a^2+1609 a^3+1592 a^4+616 a^5+100 a^6+12 a^7),\\
&p_{2}=6 a (3+134 a+5587 a^2-20016 a^3+28417 a^4-9500 a^5+2771 a^6-54 a^7+98 a^8),\\
&p_{3}=-9-927 a-10296 a^2+49352 a^3-18054 a^4-38322 a^5+8348 a^6+1596 a^7-1221 a^8-387 a^9,\\
&p_{4}=3 (-1+a) (-135-1350 a-1570 a^2+9688 a^3-2738 a^4+1838 a^5-1170 a^6+368 a^7+29 a^8),\\
&p_{5}=-3 (-1+a)^2 (2025-1251 a+1825 a^2+169 a^3-329 a^4-21 a^5+59 a^6+3 a^7),\\
&p_{6}=(-1+a)^3 (729-486 a+783 a^2-548 a^3+147 a^4-6 a^5+a^6).
\end{align*}

The last factor $f_{13}$ is of degree $57$ in $a$ and $59$ in $c$.  It is clear that the only potentially
awkward factors for a case-by-case analysis are $f_{i}$ for $i=10,11,12,13$. The curve defined by
$f_{10}(a,c)=0$ in the $ac$ plane is rational, with parametrization:
\begin{equation*}
a=-\frac{6(t-5)}{t^2+15},\quad c=-\frac{12 \left(t^2-2 t+21\right)}{t^3-3 t^2+3 t-81}.
\end{equation*}
Similarly, the curve defined by $f_{11}(a,c)=0$ is rational with parameterization:
\begin{equation*}
a=-\frac{6 (t-5)}{t^2+15},\quad c=-\frac{8 (t-3)^2}{t^3+t^2+27 t+3}.
\end{equation*}
Surprisingly, the curve defined by $f_{12}(a,c)=0$ is also of genus 0, with nontrivial rational points
when $a c(3a-2c)=0$. However, these are the only rational points. For
\[ f_{12}(a,c) = \mbox{Norm}_{\Q(\theta)/\Q} (g_0(a,c)+g_1(a,c)\theta+g_2(a,c)\theta^2), \]
where $\theta^3=2$, and
\begin{align*}
g_0(a,c)= & 562a + 60a^2 + 20a^3 - 373c + 19a c - 71a^2c - 3a^3c - 45c^2 + 47a c^2 - 3a^2c^2 + a^3c^2, \\
g_1(a,c)= & 2(223a + 24a^2 + 8a^3 - 148c + 10a c - 32a^2c - 18c^2 + 18a c^2), \\
g_2(a,c) = & 354a + 36a^2 + 12a^3 - 235c + 15a c - 48a^2c - 27c^2 + 27a c^2.
\end{align*}
Thus $f_{12}(a,c)=0$ implies $g_0=g_1=g_2=0$, and it follows easily that either $(a,c)$ is the point
at infinity $(0/0,1/0)$, or $(a,c)=(0,0),(1,3/2),(2,3)$. No new identities result. \\
We are unable to say much about the zeros of $f_{13}(a,c)=0$, a curve of degree 70. By reducing
modulo a large number of primes, we believe the curve to be absolutely irreducible. A small search
produces zeros at $(a,c)=(0,0)$, $(1,0)$, $(2,0)$, $(1/2,0)$, $(2/9,0)$, $(0,2)$, $(0,3)$, $(0,-1/23)$,
$(1,2)$, $(2,3)$, $(2,4)$, $(1,3/2)$, $(1,3/2)$, $(-2,1/3)$, $(3/2,3)$, $(8,3)$, $(1,7/3)$, but no
new identities result. \\
Case-by-case analysis now results in the following table. Solutions occur in pairs
corresponding to the transformation $T \rightarrow -T-1$, and without loss of generality
we list only one of each pair.
\begin{equation*}
\begin{array}{|c|c|c|c|}
\hline
P & Q & R & S \\
\hline
x & y & z & w \\
\hline \hline
t^2 & 1 & -2t & -1 \\
\hline
T & 1+T & T+T^2 & (1+T)^3 - t T^3 \\
\hline
& & & \\
\hline
16t^3 & 8 & -1 & 3t \\
\hline
T & 1+T & 2((1+T)^2+2t T^2) & 4T((1+T)^2+t T^2) \\
\hline
& & & \\
\hline
1 & -108 & -4 & 3 \\
\hline
1 & T & 1-3T+6T^2 & -1+6T-12T^2+18T^3 \\
\hline
& & & \\
\hline
1 & -432 & 8 & -9 \\
\hline
1 & T & 1-3T+6T^2 & -1+4T-12T^2+12T^3 \\
\hline \hline
\end{array}
\end{equation*}

\noindent
We obtain points on $P x^6+Q y^6+R z^6+S w^2=0$ by demanding $Z=z^2$. As an example, consider the
first entry in the above table, with demand $T+T^2=z^2$. The latter is parameterized by
\[ (T,z)=\left( a^2/(-a^2+b^2), \; a b/(-a^2+b^2) \right) \]
resulting in the (rather dull) identity $t^2 x^6 + y^6 -2t z^6 = w^2$ with
\[ (x,y,z,w)=(a^2, \; b^2, \; a b, \; a^6t-b^6). \]

\noindent
As second example, take the second entry in the table, where we demand
\[ (1+T)^2+2t T^2 = z^2. \]
This is parameterized by
\[ (T,z)=\left( \frac{2a b}{2t a^2-2a b-b^2}, \; \frac{2t a^2+b^2}{2t a^2-2 a b-b^2} \right) \]
resulting in the identity $2t^3 x^6+y^6-z^6+6t w^2=0$ with
\[ (x,y,z,w)=(2a b, 2t a^2-b^2, 2t a^2+b^2, 2a b(4t^2 a^4+b^4) ). \]

\noindent
The third entry (with $PQRS=\square$) demands
\[ 1-3T+6T^2 = z^2, \]
which is parameterized by
\[ (T,z)=\left( \frac{a(3a+2b)}{6a^2-b^2}, \; \frac{6a^2+3a b+b^2}{6a^2-b^2} \right), \]
resulting in the identity $x^6-108y^6-4z^6+3w^2=0$
where
\begin{align}\label{specsol}
(x,y,z,w)= & (6a^2-b^2, \; a(3a+2b), \; 6a^2+3a b+b^2, \\
& 270a^6+540a^5b+360a^4b^2+144a^3b^3+48a^2b^4+12a b^5+b^6).\notag
\end{align}
Similarly, the fourth entry leads to the identity $x^6-432y^6+8z^6=9w^2$
with
\begin{align*}
(x,y,z,w)= & (6a^2-b^2, \; a(3a+2b), \; 6a^2+3a b+b^2, \\
& 108a^6-72a^5b-216a^4b^2-144a^3b^3-42a^2b^4-8a b^5-b^6).
\end{align*}

\noindent
In the same way, we can obtain points on $P x^6+Q y^6+S m^2 z^6+R w^3=0$ by setting $(Z,W)=(w,m z^3)$,
with $m$ chosen so that the appropriate cubic curve is elliptic with positive rational rank.
As example, we consider just the first entry from the table, which leads to the cubic curve
\[ (1+T)^3 -t \; T^3 = m z^3. \]
This latter is elliptic with positive rational rank for $(t,m)=(1,6)$.
A generator is $Q(z,T)=(7/18, -37/54)$, and the resulting identity is
$x^6+y^6-36 z^6 + 2 w^3 =0$. Multiples $kQ$ correspond to solutions:
\begin{equation*}
\begin{array}{c|cccc}
k & x & y & z & w \\ \hline
1 & 37 & 17 & 21 & 629 \\ \hline
2 & 1805723 & 2237723 & 960540 & 4040707888729 \\ \hline
3 & 209143555850753 & 84691068680987 & 112490043311709 & 17712591252741962842340733211 \\ \hline
.. & .. & .. & .. & .. \\ \hline
\end{array}
\end{equation*}

\noindent
Second, we consider the equation (\ref{quadquad}) and look for polynomial solutions of the form
\begin{equation*}
X=T, \quad Y=1+T, \quad Z=b+a T+T^2, \quad W=d+c T+T^2.
\end{equation*}
If $F(X,Y,Z,W)$ denotes the left hand side in (\ref{quadquad}) then the above substitution
results in a degree six polynomial in $T$, say $F(X,Y,Z,W)=\sum_{i=0}^{6}c_{i}T^i$.
Consider the system of equations $c_{i}=0$ for $i=0,\ldots,6$.
Eliminating $a,b,c,d,S$, there results $P Q (P+R) (Q+R) H(P,Q,R)=0$, where
\begin{align*}
H(P,Q,R)=64&R^6+192(P+Q)R^5+48(5 P^2+53 P Q+5 Q^2)R^4+\\
           &32(P+Q)(5 P^2+139 P Q+5 Q^2)R^3+\\
           &12(5P^4+8P^3Q+1414P^2Q^2+8P Q^3+5 Q^4)R^2+\\
           &12(P+Q)(P^2-174 P Q+Q^2)(P^2-6 P Q+Q^2)R+\\
           &(P^2+18 P Q+Q^2)^3.
\end{align*}
Solutions with $P Q=0$ are not of interest, and rational solutions resulting from $(P+R)(Q+R)=0$
are trivial.  Consider finally $H(P,Q,R)=0$. Surprisingly, the equation defines a curve $C$ of genus
zero, with parametrization:
\begin{equation*}
P=2 v^6, \quad Q=-2,\quad R=(1-v-v^2)^3.
\end{equation*}
Consequently, we obtain just one rational solution of the system. After the substitution $T \rightarrow 1/(v T-1)$, there results the identity
\begin{equation*}
(1-T-T^2)^3+(1+T-T^2)^3=2-2T^6,
\end{equation*}
discovered by Elkies and mentioned in \cite{Elk2}.

\begin{rem}
{\rm Note that for the equation $x^6-108y^6-4z^6+3w^2=0$ we can prove much more. Indeed,
rewrite the equation in the form $(x^3-2z^3)(x^3+2z^3)=3(6y^3-w)(6y^3+w)$ and make the substitution
$w=t(x^3+2z^3)-6y^3$, leading to the equation of a genus 1 curve over $\Q(t)$:
\begin{equation}\label{speccurve}
\cal{C}:\;\left(3 t^2+1\right) x^3+\left(6 t^2-2\right) z^3-36 t y^3=0.
\end{equation}
Using the solution given by (\ref{specsol}) we can define a base change
$t=\phi(a)=(1+12 a^2+12 a^3)/(1+6 a+18 a^3)$ which give us a $\Q(a)$- rational point
on $\cal{C}_{\phi}$. This point can be used to produce infinitely many rational parametric solutions
of the equation defining $\cal{C}_{\phi}$, and thus leads to infinitely many rational curves on the
surface, say $\cal{S}$, defined by the equation $x^6-108y^6-4z^6+3w^2=0$ with $w=t(x^3+2z^3)-6y^3$.
As an immediate consequence of the existence of infinitely many rational curves on $\cal{S}$, we
get that the set $\cal{S}(\Q)$ is Zariski dense in $\cal{S}$.}
\end{rem}

Finally, we adapt the third method above to construct some surfaces of the form
\begin{equation*}
Px^6+Qy^6+Rz^6+Sw^2=0
\end{equation*}
with infinitely many rational points. Consider the genus ten curve $C:\;Px^6+Qy^6+Rz^6=0$ and suppose
that $(a,b,c)$, with $abc\neq 0$ is a rational point lying on $C$. Put
\begin{equation*}
x=aT,\quad y=bT-\frac{c^3R}{b^3Q},\quad z=cT+1,
\end{equation*}
 Then, with $x, y, z$ defined above we get
\begin{equation*}
G(T)=\frac{R(b^6Q+c^6R)}{b^{30}Q^5}\sum_{i=0}^{3}C_{i}T^{i},
\end{equation*}
where
\begin{align*}
&C_{0}=b^{24} Q^4 - b^{18} c^6 Q^3 R + b^{12} c^{12} Q^2 R^2 - b^6 c^{18} Q R^3+c^{24}R^4,\\
&C_{1}=-6 b^6 c Q (-b^6 Q + c^6 R) (b^{12} Q^2 + c^{12} R^2),\\
&C_{2}=15 b^{12} c^2 Q^2 (b^{12} Q^2 - b^6 c^6 Q R + c^{12} R^2),\\
&C_{3}=-20 b^{18}c^3 Q^3 (-b^6 Q + c^6 R),\\
&C_{4}=15 b^{24}c^4Q^4.
\end{align*}
Now $G$ is of degree four in $T$ and thus if a rational number $S$ is chosen so that the quartic curve
\begin{equation*}
\cal{C}:\;-S w^2=G(T)
\end{equation*}
has infinitely many rational points, then there will automatically be infinitely many rational points
on the surface $Px^6+Qy^6+Rz^6+Sw^2=0$. As before, we can guarantee positive rank by taking $S=-G(t)$,
for a parameter $t$.\\
As an example, take $(a,b,c)=(1,1,1)$, $(P,Q,R)=(p,q,-p-q)$. It is straightforward to determine that
in consequence there are infinitely many points on either of the two surfaces
\begin{align*}
p x^6+q y^6 -(p+q) z^6 -15 p q(p+q)w^2 & = 0, \\
p x^6+q y^6 -(p+q) z^6 -p q(p^5+q^5) w^2 & = 0.
\end{align*}
If we specialize further to $(P,Q,R)=(1,1,-2)$, then we get the quartic curve
$-S w^2=2(31+90T+105T^2+60T^3+15T^4)$. Taking now $S=-2(31+90t+105t^2+60t^3+15t^4)$ we consider the curve
\begin{equation*}
\cal{C}: \; (31+90t+105t^2+60t^3+15t^4)w^2=31+90T+105T^2+60T^3+15T^4.
\end{equation*}
Taking the point $(T,w)=(t,1)$ as a point at infinity, then the point $Q(T,w)=(t,-1)$
is of infinite order, and its multiples pull back to infinitely many points on the surface.
For example, $2Q$ delivers the point, on simplifying by replacing $t$ by $-1-t$:
\begin{align*}
(x,y,z,w) = (& -1 + 3t + 60t^3 + 90t^4 + 90t^5 + 900t^6 + 675t^8 - 225t^9, \\
             & -1 - 3t - 60t^3 + 90t^4 - 90t^5 + 900t^6 + 675t^8 + 225t^9, \\
             & 3t(1 + 20t^2 + 30t^4 - 75t^8), \\
&  (-1 + 90t^4 + 900t^6 + 675t^8)(1 + 60t^2 + 420t^4 + 6300t^6 + 55350t^8 + \\
& 94500t^{10} + 94500t^{12} + 202500t^{14} + 50625t^{16}))
\end{align*}
on the surface
\[ x^6+y^6-2 z^6 = 2(1+15 t^2+15 t^4)w^2, \]
etc.

\section{Final remarks} \label{sec6}

We can try to extend the polynomial method of the last two sections in the following manner.
Seek $P,Q,R,S \in \Q$ and polynomials $G_{i}, F\in\Q[t]$, with
$\op{deg}G_{i}\leq 2$ for $i=1,2,3$ and $\op{deg}F\leq 4$, such that either
\begin{equation}
\label{quartsurf}
PG_{1}(t)^4+QG_{2}(t)^4+RG_{3}(t)^4+SF(t)^2=0,
\end{equation}
or
\begin{equation}
\label{sextsurf}
PG_{1}(t)^6+QG_{2}(t)^6+RG_{3}(t)^6+SF(t)^3=0.
\end{equation}
If we can find such polynomials and numbers $P,Q,R,S$, the existence
of a rational point, say $(t,w)$, on the quartic curve
\begin{equation*}
C:\;w^2=F(t),
\end{equation*}
then implies respectively the existence either of a rational point on the surface $Px^4+Qy^4+Rz^4+Sw^4=0$,
or a rational point on the surface $Px^6+Qy^6+Rz^6+Sw^6=0$, in each case with
$(x,y,z,w)=(G_{1}(t),G_{2}(t),G_{3}(t),w)$. \\ \\
The system of equations underlying (\ref{quartsurf}) is too large for complete analysis.
However, there seem to be numerous such identities, and we list just one of the simpler ones:
\[(P,Q,R,S) = (a^2+2a b+2b^2-2c, 4(a^2+2a b+2b^2 - 2c), (a^2-2a b+2b^2)c^2, -a^2+2a b-2b^2), \]
\begin{align*}
(G_1,G_2,G_3)= & (2a(2a b-c) - 4b(2a b-c)t + (4b^3+(a-2b)c)t^2, \\
& 2(a^3-(a-b)c) - 2a(2a b-c)t + b(2a b-c)t^2,  \\
& 2(2a b-c) - 2(a^2+2b^2-c)t + (2a b-c)t^2 ),
\end{align*}
with
\begin{align*}
F(t) & =4\left( 2a^4(a^2+2ab+2b^2) -2a^3(3a+2b)c + (5a^2-2ab+2b^2)c^2 - c^3 \right) - \\
& 8(2ab - c) \left( 2a^2(a^2+2ab+2b^2) - 2a(2a+b)c + c^2 \right)t + \\
& 8(2ab - c) \left( 3ab(a^2+2ab+2b^2) - (a^2+6ab+2b^2)c + c^2 \right)t^2 - \\
& 4(2ab - c) \left( 4b^2(a^2+2ab+2b^2) -2b(a+4b)c + c^2 \right)t^3 + \\
& \left( 8b^4(a^2+2ab+2b^2) -8b^3(a+3b)c + (a^2-2ab+10b^2)c^2 - c^3 \right)t^4.
\end{align*}
The generic Picard rank is 2.\\ \\
In the case of (\ref{sextsurf}), we offer the identity

\[\begin{array}{l}
2\left( (p+q)t^2+4qt+p-q \right)^6-2\left( 2qt^2+(2p-2q)t-p+3q \right)^6\\[0.1in]
\;\;-(p^2-5q^2)^3(t^2-t-1)^6-\left( (p^2+4pq-q^2)t^4+(2p^2+22q^2)t^3 \right. \\[0.1in]
\;\;-(3p^2-24pq+9q^2)t^2+
\left. (6p^2-16pq+18q^2)t+8pq-11q^2-p^2 \right)^3=0,\end{array} \]
which, on taking $p=1,\,q=1,$ yields the identity
\[ 2(t^2+2t)^6 - 2(t^2+1)^6 + (t^2-t-1)^6 = (t^4+6 t^3+3t^2+2t-1)^3. \]
For generic $t_0$, the elliptic quartic
\[ t^4+6t^3+3t^2+2t-1 = (t_0^4+6t_0^3+3t_0^2+2t_0-1)W^2 \]
with $(t,W)=(t_0,1)$ as origin, has point $Q(t,W)=(t_0,-1)$ of infinite order.
Pullbacks of $mQ$, $m=2,3,...$, now give infinitely many solutions of
\[ 2X^6-2Y^6+Z^6 = (t_0^4+6t_0^3+3t_0^2+2t_0-1)^3W^6. \]
For example, $2Q$ delivers $(X,Y,Z,W)=$
\begin{align*}
(t(2+t)(-1+2t+2t^2+ 14t^3+20t^4+32t^5+26t^6-4t^7+2t^8) \times & \\
(3+6t-6t^2-6t^3+24t^5+18t^6+12t^7+2t^8), & \\
(1+t^2)(1+8t^2-68t^4-192t^5+296t^6+1488t^7+2560t^8+2112t^9+656t^{10} & \\
-144t^{11}+292t^{12}+480t^{13}+248t^{14}+48t^{15}+4t^{16}), & \\
(1+t-t^2)(1-4t-10t^2-60t^3+50t^4+112t^5+412t^6+600t^7+1540t^8+2640t^9 & \\
+3640t^{10}+3440t^{11}+2240t^{12}+960t^{13}+320t^{14}+64t^{15}+4t^{16}), & \\
1+4t-6t^2+28t^3+126t^4+336t^5+420t^6+104t^7-732t^8-1104t^9-1176t^{10} & \\
-1008t^{11}-784t^{12}-448t^{13}-96t^{14}+4t^{16}). &
\end{align*}
Accordingly, observe that we deduce
infinitely many integer solutions of the Diophantine equation
\[ U_1^6+2 V_1^6+2 W_1^6 = U_2^6+2 V_2^6+2 W_2^6; \]
and more generally, integer solutions of the arbitrarily long chain
\[ 2 X_1^6-2 Y_1^6+Z_1^6 \; = \; 2 X_2^6-2 Y_2^6+Z_2^6 \; = \; 2 X_3^6-2 Y_3^6+Z_3^6 \; = \; \dots \]
No further solution to (\ref{sextsurf}) over $\Q$ was discovered. However,
the following curious identity came to light:
\begin{equation*}
125u^6 - 2(u^2+1)^6 + 2(\epsilon u^2-\epsilon^{-1})^6 = (2\epsilon u^4-3u^2-2\epsilon^{-1})^3,
\end{equation*}
where $\epsilon = (1+\sqrt{5})/2$. \\ \\

\bigskip

\noindent School of Mathematics and Statistics, Arizona State University, Tempe AZ 85287-1804,
USA; email:\;{\tt bremner@asu.edu}

\bigskip

\noindent Foreign Service Institute, Old J.N.U. Campus, Baba Gang Nath Marg, New Delhi -– 110 067, India; email:\;{\tt ajaic203@yahoo.com}

\bigskip

\noindent Maciej Ulas, Jagiellonian University, Institute of Mathematics,
{\L}ojasiewicza 6, 30-348 Krak\'ow, Poland; email:\;{\tt maciej.ulas@uj.edu.pl}

\end{document}